\documentclass[12pt]{amsart}
\usepackage{amsmath}
\usepackage{amsfonts}
\makeatletter
\def\@seccntformat#1{%
  \protect\textup{%
    \protect\@secnumfont
    \expandafter\protect\csname format#1\endcsname 
    \csname the#1\endcsname
    \protect\@secnumpunct
  }%
}

\begin{document}
\title[Invariant Theory in Exterior Algebras]
{Invariant Theory in Exterior Algebras and Amitsur--Levitzki type theorems}
\author[M. Itoh]
{Minoru Itoh}
\date{}
\address{Department of Mathematics and Computer Science, 
          Faculty of Science,
          Kagoshima University, Kagoshima 890-0065, Japan}
\email{itoh@sci.kagoshima-u.ac.jp }
\begin{abstract}
   This article discusses invariant theories in some exterior algebras,
   which are closely related to Amitsur--Levitzki type theorems.

   First we consider the exterior algebra on the vector space of square matrices of size $n$,
   and look at the invariants under conjugations.
   We see that the algebra of these invariants is isomorphic to the exterior algebra 
   on an $n$-dimensional vector space.
   Moreover we give a Cayley--Hamilton type theorem for these invariants
   (the anticommutative version of the Cayley--Hamilton theorem).
   This Cayley--Hamilton type theorem can also be regarded as a refinement of 
   the Amitsur--Levitzki theorem.

   We discuss two more Amitsur--Levitzki type theorems
   related to invariant theories in exterior algebras.
   One is a famous Amitsur--Levitzki type theorem due to Kostant and Rowen,
   and this is related to $O(V)$-invariants in $\Lambda(\Lambda_2(V))$.
   The other is a new Amitsur--Levitzki type theorem, 
   and this is related to $GL(V)$-invariants in $\Lambda(\Lambda_2(V) \oplus S_2(V^*))$.
\end{abstract}
\thanks{This research was partially supported by JSPS Grant-in-Aid for Young Scientists (B) 20740020 and JSPS Grant-in-Aid for Young Scientists (B) 24740021.
}
\maketitle
\theoremstyle{plain}
   \newtheorem{theorem}{Theorem}[section]
   \newtheorem{proposition}[theorem]{Proposition}
   \newtheorem{lemma}[theorem]{Lemma}
   \newtheorem{corollary}[theorem]{Corollary}
\theoremstyle{remark}
   \newtheorem*{remark}{Remark}
   \newtheorem*{remarks}{Remarks}
\numberwithin{equation}{section}
\newcommand{\mybinom}[2]{\left(\!{#1 \atop #2}\!\right)}
\newcommand{\bibinom}[2]{\left(\!\!\left(\!{#1 \atop #2}\!\right)\!\!\right)}
%

%
\section*{Introduction}
%
%
In this article,
we discuss invariant theory in exterior algebras on some matrix spaces,
and give several Cayley--Hamilton type relations for invariants in these exterior algebras
as consequences of the second fundamental theorem of invariant theory for vector invariants.
These Cayley--Hamilton type relations are all closely related to Amitsur--Levitzki type theorems.

\subsection{}
We first consider $GL(V)$-invariants in the exterior algebra $\Lambda(V \otimes V^*)$,
where $V$ is an $n$-dimensional complex vector space, and $V^*$ is its linear dual.
The algebra $\Lambda(V \otimes V^*)^{GL(V)}$ of these invariants is isomorphic to
the exterior algebra on an $n$-dimensional vector space.
Indeed $\Lambda(V \otimes V^*)^{GL(V)}$ is generated by the following $n$ elements,
and these $n$ generators have no relations besides anticommutativity (Theorem~\ref{thm:FFT_and_SFT1}):
$$
   \operatorname{tr}(X^1),
   \operatorname{tr}(X^3),
   \ldots,
   \operatorname{tr}(X^{2n-1}).
$$
Here we put $X = (x_{ij})_{1 \leq i,j \leq n} \in \operatorname{Mat}_{n,n}(\Lambda(V \otimes V^*))$,
where $x_{ij}$ is the standard basis of $V \otimes V^*$.
This result is similar to the fact that the algebra of the $GL(V)$-invariants in 
the polynomial algebra on $V \otimes V^*$ 
is isomorphic to the polynomial algebra in $n$ variables.

We also give the following Cayley--Hamilton type theorem for these generators (Theorem~\ref{thm:CH_type_thm1}):
$$
   n X^{2n-1} 
   - \operatorname{tr}(X^1) X^{2n-2} 
   - \operatorname{tr}(X^3) X^{2n-4}
   - \cdots 
   - \operatorname{tr}(X^{2n-3}) X^2 
   - \operatorname{tr}(X^{2n-1}) X^0
   = 0.
$$
We can regard this as the anticommutative version of the Cayley--Hamilton theorem.
From this, the following Amitsur--Levitzki theorem is immediate:
\begin{equation}\label{eq:AL_thm}
   \sum_{\sigma \in S_{2n}} 
   \operatorname{sgn}(\sigma) 
   X_{\sigma(1)} X_{\sigma(2)} \cdots X_{\sigma(2n)} = 0.
\end{equation}
Here $X_1,\ldots,X_{2n}$ are complex square matrices of size $n$.
In this sense, we can regard our Cayley--Hamilton type theorem 
as a refinement of the Amitsur--Levitzki theorem (\ref{eq:AL_thm}).

\begin{remark}
   This Cayley--Hamilton type theorem in 
   $\operatorname{Mat}_{n,n}(\Lambda (V \otimes V^{*}))$ 
   was also given independently by \cite{BPS} (see also \cite{DPP})
   as a consequence of the ordinary Cayley--Hamilton theorem.
   Moreover, Procesi discussed the Amitsur--Levitzki theorem 
   with this Cayley--Hamilton type theorem in~\cite{Pr}.

   In spite of an intersection with these papers,
   the author wrote the proofs of results 
   for $GL(V)$-invariants in $\Lambda(V \otimes V^*)$
   in Sections~\ref{sec:inv_theory1} and~\ref{sec:CH_type_thm1} of this article,
   because these can be regarded as the prototype for the study of
   $O(V)$-invariants in $\Lambda(\Lambda_2(V))$
   and $GL(V)$-invariants in $\Lambda(\Lambda_2(V) \oplus S_2(V^*))$
   in Sections~\ref{sec:AL_type_thm_due_to_Kostant} and~\ref{sec:new_AL_type_thm}.
   These results are all similarly deduced from
   the first and the second fundamental theorems
   of invariant theory for vector invariants.
\end{remark}

\subsection{}
We also discuss the following Amitsur--Levitzki type theorem due to Kostant \cite{K1}
and Rowen \cite{Row1}:
\begin{equation}\label{eq:Kostant}
   \sum_{\sigma \in S_{2n-2}} 
   \operatorname{sgn}(\sigma) 
   A_{\sigma(1)}  A_{\sigma(2)} \cdots A_{\sigma(2n-2)} = 0.
\end{equation}
Here $A_1,\ldots,A_{2n-2}$ are complex alternating matrices of size $n$.

The proof of this theorem (\ref{eq:Kostant}) is much more difficult 
than that of (\ref{eq:AL_thm}).
Kostant first proved this theorem using theory of cohomology of Lie algebras, when $n$ is even \cite{K1}.
Later, Rowen gave an elementary but technical proof for arbitrary $n$ \cite{Row1}.

In this article,
we give a new proof of (\ref{eq:Kostant}) 
through the relation to invariant theory in an exterior algebra.
Namely this theorem (\ref{eq:Kostant}) is related to
$O(V)$-invariants in the exterior algebra $\Lambda(\Lambda_2(V))$
on the second antisymmetric tensor $\Lambda_2(V)$ of $V$,
where $V$ is an $n$-dimensional complex vector space with nondegenerate symmetric bilinear form.
The algebra $\Lambda(\Lambda_2(V))^{O(V)}$ is generated by the following elements,
and these elements have no relations besides anticommutativity (Theorem~\ref{thm:FFT_and_SFT2}):
\begin{align*}
   \operatorname{tr}(A^3), \operatorname{tr}(A^7), \operatorname{tr}(A^{11}), 
   \ldots, \operatorname{tr}(A^{4m-5}), \qquad n&=2m, \\
   \operatorname{tr}(A^3), \operatorname{tr}(A^7), \operatorname{tr}(A^{11}), 
   \ldots, \operatorname{tr}(A^{4m-1}), \qquad n&=2m+1.
\end{align*}
Here we put $A = (a_{ij})_{1 \leq i,j \leq n} \in \operatorname{Mat}_{n,n}(\Lambda(\Lambda_2(V))$,
where $a_{ij}$ is the standard basis of $\Lambda_2(V)$. 
For these generators,
we also give a Cayley--Hamilton type theorem (Theorem~\ref{thm:CH_type_thm2})%
\footnote{Shortly after the post of the first version of this article to arXiv,
S. Dolce posted the first version of \cite{D} to arXiv.
Dolce studied $G$-invariants in $\Lambda(\Lambda_2(V))$ and $\Lambda(S_2(V))$,
and gave Cayley--Hamilton type theorems
in $\operatorname{Mat}_n(\Lambda(\Lambda_2(V)))$ and 
$\operatorname{Mat}_n(\Lambda(S_2(V)))$,
where $G$ is the symplectic group or the odd orthogonal group.
His results contain the case $n=2m+1$ of our Theorem~\ref{thm:CH_type_thm2}.}:
\begin{align*}
   (n-2)A^{2n-3} - \sum_{0 \leq k \leq m-2} \operatorname{tr}(A^{4k+3}) A^{2n-3-4k-3} &= 0,
   \qquad
   n=2m, \\
   nA^{2n-3} - \sum_{0 \leq k \leq m-1} \operatorname{tr}(A^{4k+3}) A^{2n-3-4k-3} &= 0,
   \qquad
   n=2m+1.
\end{align*}
The Amitsur--Levitzki type theorem (\ref{eq:Kostant})
is immediate from this.
Namely this Cayley--Hamilton type theorem can be regarded as a refinement of (\ref{eq:Kostant}). 

\subsection{}
Moreover we give the following Amitsur--Levitzki type theorem 
(Theorem~\ref{thm:new_AL_type_thm}):
\begin{equation}\label{eq:new_AL_type_thm}
   \sum_{\sigma \in S_n, \,\, \tau \in S_{n-1}} 
   \operatorname{sgn}(\sigma) \operatorname{sgn}(\tau)
   A_{\sigma(1)} B_{\tau(1)} A_{\sigma(2)} B_{\tau(2)} \cdots 
   A_{\sigma(n-1)} B_{\tau(n-1)} A_{\sigma(n)} = 0.
\end{equation}
Here $A_1,\ldots,A_n$ are complex alternating matrices of size $n$, and 
$B_1,\ldots,B_{n-1}$ are complex symmetric matrices of size $n$.

This new Amitsur--Levitzki type theorem is related to
invariant theory in the exterior algebra $\Lambda(\Lambda_2(V) \oplus S_2(V^*))$
on the direct product of the second antisymmetric tensor $\Lambda_2(V)$ of $V$ and 
the second symmetric tensor $S_2(V^*)$ of $V^*$,
where $V$ is an $n$-dimensional complex vector space.
For this exterior algebra, we give two results. 
First, we do not have nontrivial $GL(V)$-invariants
(Theorem~\ref{thm:FFT_and_SFT3}):
$$
   \Lambda(\Lambda_2(V) \oplus S_2(V^*))^{GL(V)} = \mathbb{C} 1.
$$
Secondly we have the following relation (Theorem~\ref{thm:CH_type_thm3}):
$$
   (AB)^{n-1}A = 0.
$$
Here 
we put 
$A = (a_{ij})_{1 \leq i,j \leq n}$ and
$B = (b_{ij})_{1 \leq i,j \leq n} \in \operatorname{Mat}_{n,n}(\Lambda(\Lambda_2(V) \oplus S_2(V^*)))$,
where 
$a_{ij}$ and $b_{ij}$ are the standard bases of $\Lambda_2(V)$ and $S_2(V^*)$, respectively. 
We can regard this relation as a Cayley--Hamilton type theorem,
and the Amitsur--Levitzki type theorem (\ref{eq:new_AL_type_thm}) follows from this.

\subsection{}
In this article, we deal with the following algebras of invariants in exterior algebras:
$$
   \Lambda(V \otimes V^*)^{GL(V)}, \quad
   \Lambda(\Lambda_2(V))^{O(V)}, \quad
   \Lambda(\Lambda_2(V))^{SO(V)}, \quad
   \Lambda(\Lambda_2(V) \oplus S_2(V^*))^{GL(V)}.
$$
Each of these algebras is isomorphic to an exterior algebra.
We will see this fact by giving the generators of these algebras explicitly
(as consequences of the first fundamental theorems of invariant theory for vector invariants).
We note the relation of this fact with cohomology theory of Lie algebras.
Namely, for $\Lambda(V \otimes V^*)^{GL(V)}$ and $\Lambda(\Lambda_2(V))^{SO(V)}$,
this fact also follows from cohomology theory of Lie algebras.
Indeed, for a reductive Lie algebra $\mathfrak{g}$,
the algebra $\Lambda(\mathfrak{g}^*)^{\mathfrak{g}}$ is isomorphic to the cohomology ring
$H(\mathfrak{g})$, 
and this is known to be isomorphic to an exterior algebra (see \cite{M}).

%
\section{Invariant theory for $GL(V)$-invariants in $\Lambda(V \otimes V^*)$}
\label{sec:inv_theory1}
%
%
First in this section, we study invariant theory
in the exterior algebra on the vector space of square matrices.
Let $V$ be an $n$-dimensional complex vector space,
and $V^*$ be its linear dual.
The general linear group $GL(V)$ naturally acts on $V \otimes V^*$
and moreover the exterior algebra $\Lambda(V \otimes V^*)$ on $V \otimes V^*$.
Let us study the algebra $\Lambda(V \otimes V^*)^{GL(V)}$ 
of $GL(V)$-invariants in $\Lambda(V \otimes V^*)$.

Consider the following element in $\Lambda(V \otimes V^*)$,
where $x_{ij}$ is the standard basis of $V \otimes V^*$:
$$
   q_k 
   = \sum_{1 \leq i_1,\ldots,i_k \leq n} 
   x_{i_1i_2} \wedge x_{i_2i_3} \wedge \cdots \wedge x_{i_ki_1}.
$$
From now on, we omit the symbol ``$\wedge$,'' so that
$$
   q_k 
   = \sum_{1 \leq i_1,\ldots,i_k \leq n} x_{i_1i_2} x_{i_2i_3} \cdots x_{i_ki_1}.
$$
Moreover we can express this as $q_k = \operatorname{tr}(X^k)$ using the matrix 
$$
   X = (x_{ij})_{1 \leq i,j \leq n} \in \operatorname{Mat}_{n,n}(\Lambda(V \otimes V^*)).
$$

\begin{proposition}\label{thm:tr(X^2r)=0}\sl
   We have $q_k = 0$ for $k = 2,4,6,\ldots$.
\end{proposition}

\begin{proof}
   This is immediate from the following calculation:
   $$
      q_{2r}
      = \sum_{1 \leq i_1,\ldots,i_{2r} \leq n} 
      x_{i_1i_2} x_{i_2i_3} \cdots x_{i_{2r}i_1} 
      = - \sum_{1 \leq i_1,\ldots,i_{2r} \leq n} 
      x_{i_2i_3} \cdots x_{i_{2r}i_1} x_{i_1i_2} 
      = - q_{2r}.
   $$
   Here we moved $x_{i_1i_2}$ at the left end to the right end in the second equality.
\end{proof}

\begin{proposition}\sl
   $q_k$ is $GL(V)$-invariant.
\end{proposition}

\begin{proof}
   This is immediate from the relation $q_k = \operatorname{tr}(X^k)$ and the following equality:
   $$
      \pi(g) X = (\pi(g)x_{ij})_{1 \leq i,j \leq n}
      = {}^t\!g X \,{}^t\!g^{-1}.
   $$
   Here we denote by $\pi$ the natural action of $GL(V)$ on $\Lambda(V \otimes V^*)$.
\end{proof}

As the first and second fundamental theorems of invariant theory, we have the following theorem:

\begin{theorem}\label{thm:FFT_and_SFT1}\sl
   The algebra $\Lambda(V \otimes V^*)^{GL(V)}$ is generated by
   $q_1,q_3,\ldots,q_{2n-3},q_{2n-1}$.
   Moreover these generators are anticommuting with each other,
   and have no other relations besides this anticommutativity.
   Namely the following forms a linear basis of 
   $\Lambda(V \otimes V^*)^{GL(V)}${\rm :}
   \begin{equation}\label{eq:basis_of_invariants}
      \{ q_{k_1} \cdots q_{k_d} \,|\,
      \text{$k_1,\ldots,k_d${\rm :} odd}, \,\, 
      0 < k_1 < \cdots < k_d < 2n, \,\,
      d = 0,1,\ldots,n \}.
   \end{equation}
   Thus $\Lambda(V \otimes V^*)^{GL(V)}$
   is isomorphic to the exterior algebra on an $n$-dimensional vector space.
\end{theorem}

Let us prove this.
Actually we prove the following three propositions:

\begin{proposition}\label{prop:weak_FFT1}\sl
   The algebra $\Lambda(V \otimes V^*)^{GL(V)}$ is generated by $q_1,q_3,\ldots$.
\end{proposition}

\begin{proposition}\label{prop:strong_FFT1}\sl
   We have $q_{2n+1} = q_{2n+3} = \cdots = 0$.
\end{proposition}

\begin{proposition}\label{prop:SFT1}\sl
   The elements $q_1,q_3,\ldots,q_{2n-1}$ are anticommuting with each other,
      and have no other relations besides this anticommutativity.
\end{proposition}

\begin{proof}[Proof of Proposition~{\sl\ref{prop:weak_FFT1}}]
   We consider the homogeneous decomposition 
   $$
      \Lambda(V \otimes V^*) = \bigoplus_{k=0}^{n^2}\Lambda_k(V \otimes V^*).
   $$
   This is a decomposition as $GL(V)$-spaces,
   so that it suffices to describe the $GL(V)$-invariants 
   in $\Lambda_k(V \otimes V^*)$.
   The following map $(V \otimes V^*)^{\otimes k} \to \Lambda_k(V \otimes V^*)$
   is a surjective homomorphism of $GL(V)$-spaces:
   $$
      e_{i_1} \otimes e^*_{j_1} \otimes \cdots \otimes e_{i_k} \otimes e^*_{j_k}
      \mapsto
      x_{i_1j_1} \cdots x_{i_kj_k}.
   $$
   Thus any $GL(V)$-invariant in $\Lambda_k(V \otimes V^*)$
   comes from a $GL(V)$-invariants in $(V \otimes V^*)^{\otimes k}$.
   By the first fundamental theorem of invariant theory for vector invariants (\cite{W}, \cite{GW}), 
   any $GL(V)$-invariant in $(V \otimes V^*)^{\otimes k}$ 
   can be expressed as a linear combination of elements in the form
   $$
      \sum_{1 \leq i_1,\ldots,i_k \leq n}
      e_{i_1} \otimes e^*_{i_{\sigma(1)}} \otimes \cdots \otimes e_{i_k} \otimes e^*_{i_{\sigma(k)}}
   $$
   with $\sigma \in S_k$.
   The image of this element is equal to 
   $$
      \sum_{1 \leq i_1,\ldots,i_k \leq n}
      x_{i_1 i_{\sigma(1)}} \cdots x_{i_k i_{\sigma(k)}},
   $$
   and this is equal to a product of $q_1,q_3,q_5,\ldots$ up to a sign.
   Thus any $GL(V)$-invariant in $\Lambda_k(V)$ is expressed as a linear combination 
   of products of $q_1,q_3,q_5,\ldots$.
\end{proof}

Proposition~\ref{prop:strong_FFT1} will follow from a Cayley--Hamilton type theorem in the next section.

Finally let us prove Proposition~\ref{prop:SFT1}.
First, $q_1,q_3,\ldots,q_{2n-1}$ are anticommuting with each other,
because these are all odd elements.
Moreover, to prove the linear independence of (\ref{eq:basis_of_invariants}),
it suffices to show that $q_1 q_3 \cdots q_{2n-3} q_{2n-1}$ is nonzero.
To show this, 
we look at the following element in $\Lambda(V \otimes V^*)$
(the product of all entries of the matrix $X$):
$$
   h =
   x_{11} x_{12} \cdots x_{1n} \cdot
   x_{21} x_{22} \cdots x_{2n} \cdot
   \cdots \cdot
   x_{n1} x_{n2} \cdots x_{nn}. 
$$
\begin{lemma}\label{lemma:GL-invariance_of_h}\sl
   $h$ is $GL(V)$-invariant.
\end{lemma}

This is easily seen from the following general fact:

\begin{lemma}\label{lemma:invariance_of_the_highest_element}\sl
   Let $W$ be an $N$-dimensional complex vector space,
   and consider the natural action $\pi$ of $GL(W)$ on $\Lambda(W)$.
   Then we have
   $$
      \pi(g) e_1 \cdots e_N = \det(g)^N e_1 \cdots e_N
   $$
   for $g \in GL(W)$.
   Here $e_1,\ldots,e_N$ are a basis of $W$.
\end{lemma}

\begin{proof}[Proof of Lemma~{\sl\ref{lemma:GL-invariance_of_h}}]
   We have $\det\rho(g) = 1$,
   where $\rho$ is the natural action of $GL(V)$ on $V \otimes V^*$.
   The assertion is immediate from this and Lemma~\ref{lemma:invariance_of_the_highest_element}.
\end{proof}

Now Proposition~\ref{prop:SFT1} is proved as follows:

\begin{proof}[Proof of Proposition~{\sl\ref{prop:SFT1}}]
   Since $h$ is $GL(V)$-invariant,
   this is generated by $q_1,q_3,\ldots,q_{2n-1}$.
   Moreover, $h$ is equal to $q_1 q_3 \cdots q_{2n-1}$ up to constant,
   because the degree of $h$ is $n^2$ and $1+3+\cdots+(2n-1) = n^2$.
   Thus $q_1 q_3 \cdots q_{2n-1}$ cannot be equal to $0$.
\end{proof}

%
\section{Cayley--Hamilton type theorem for $\Lambda(V \otimes V^*)$}
\label{sec:CH_type_thm1}
%
%
We have a Cayley--Hamilton type theorem for $q_{2k+1} = \operatorname{tr}(X^{2k+1})$ and the matrix $X$:

\begin{theorem}\label{thm:CH_type_thm1}\sl
   We have the following relation in
   $\operatorname{Mat}_{n,n}(\Lambda(V\otimes V^*))$
   {\rm (}here $X^0$ means the unit matrix{\rm )}{\rm :}
   $$
      n X^{2n-1} 
      - \operatorname{tr}(X^1) X^{2n-2} 
      - \operatorname{tr}(X^3) X^{2n-4}
      - \cdots 
      - \operatorname{tr}(X^{2n-3}) X^2 
      - \operatorname{tr}(X^{2n-1}) X^0
      = 0.
   $$
\end{theorem}

To prove this, we introduce a notation for alternating sums.
Fix a $\mathbb{C}$-algebra $R$,
and consider matrices $\Omega_1,\ldots,\Omega_r \in \operatorname{Mat}_{n,n}(R)$
and two column vectors $\alpha = \,{}^t(\alpha_1,\ldots,\alpha_n)$, 
$\beta = \,{}^t(\beta_1,\ldots,\beta_n) \in \operatorname{Mat}_{n,1}(R)$.
For these, we put
\begin{align}
   \label{eq:def_of_D1}
   D(\Omega_1,\ldots,\Omega_r)  
   &= \sum_{\sigma \in S_r} \sum_{1 \leq i_1,\ldots,i_r \leq n} 
   \operatorname{sgn}(\sigma)
   (\Omega_1)_{i_1 i_{\sigma(1)}} \cdots (\Omega_r)_{i_r i_{\sigma(r)}}, 
   \allowdisplaybreaks\\
   \label{eq:def_of_D2}
   D(\Omega_1,\ldots,\Omega_r \,|\, \alpha \,|\, \beta)  
   &= \sum_{\sigma \in S_{r+1}} \sum_{1 \leq i_1,\ldots,i_{r+1} \leq n} 
   \operatorname{sgn}(\sigma)
   (\Omega_1)_{i_1 i_{\sigma(1)}} \cdots (\Omega_r)_{i_r i_{\sigma(r)}} 
   \alpha_{i_{r+1}} \beta_{i_{\sigma(r+1)}}.
\end{align}
We denote the $r$ repetition of the matrix $\Omega$ simply by $[\Omega]^r$.
For example, we have
$$
   D([\Omega]^r,\Phi \,|\, \alpha \,|\, \beta)  
   = D(\underbrace{\Omega,\ldots,\Omega}_r,\Phi \,|\, \alpha \,|\, \beta).
$$

Under this notation, we put
\begin{align*}
   Q &= D([X^2]^{n-1},X \,|\, v \,|\, w) \\
     &= \sum_{\sigma \in S_{n+1}} \sum_{1 \leq i_1,\ldots,i_{n+1} \leq n} 
   \operatorname{sgn}(\sigma)
   (X^2)_{i_1 i_{\sigma(1)}} \cdots (X^2)_{i_{n-1} i_{\sigma(n-1)}} 
   X_{i_n i_{\sigma(n)}} v_{i_{n+1}} w_{i_{\sigma(n+1)}}.
\end{align*}
Here $v_1,\ldots,v_n$, $w_1,\ldots,w_n$ are arbitrary complex numbers,
and we put $v = {}^t(v_1,\ldots,v_n)$ and $w = {}^t(w_1,\ldots,w_n)$.
On one hand, this $Q$ is equal to $0$.
Indeed we have
$$
   \sum_{\sigma \in S_{n+1}} 
   \operatorname{sgn}(\sigma)
   (X^2)_{i_1 i_{\sigma(1)}} \cdots (X^2)_{i_{n-1} i_{\sigma(n-1)}} 
   X_{i_n i_{\sigma(n)}} v_{i_{n+1}} w_{i_{\sigma(n+1)}} = 0
$$
for any $1 \leq i_1,\ldots,i_{n+1} \leq n$,
because $i_1,\cdots,i_{n+1}$ cannot be distinct.
On the other hand, we can express $Q$ as follows:

\begin{proposition}\label{prop:equivalent_with_CH1}\sl
   We have
   $$
      Q = 
      (-)^n \{ 
      n! \,{}^tw X^{2n-1} v 
      - (n-1)! \sum_{0 \leq k \leq n-1} \operatorname{tr}(X^{2k+1}) \,{}^tw X^{2n-2-2k} v 
      \}.
   $$
\end{proposition}

Theorem~\ref{thm:CH_type_thm1} is immediate from this.

To prove this proposition,
we use the following recurrence relations
(these can be regarded as kinds of the Laplace expansion):

\begin{lemma}\label{lemma:recurrence_relations1}\sl
   We have
   \begin{align*}
      D([X^2]^r, X^s)
      &= D([X^2]^r) \operatorname{tr}(X^s)
      - r D([X^2]^{r-1},X^{s+2}), 
      \allowdisplaybreaks\\
      D([X^2]^r \,|\, X^s v \,|\, w)
      &= D([X^2]^r) \,{}^tw X^s v
      - r D([X^2]^{r-1} \,|\, X^{s+2} v \,|\, w), 
      \allowdisplaybreaks\\
      D([X^2]^r, X \,|\, X^s v \,|\, w)
      &= D([X^2]^r, X) \,{}^tw X^s v
      - D([X^2]^r \,|\, X^{s+1} v \,|\, w) \\
      &\qquad
      - r D([X^2]^{r-1}, X \,|\, X^{s+2} v \,|\, w).
   \end{align*}
\end{lemma}

\begin{proof}
   Let us prove the last relation.
   For $\sigma \in S_{r+2}$, we put
   $$
      D_{\sigma} = \sum_{1 \leq i_1,\ldots,i_{r+2} \leq n} 
      (X^2)_{i_1 i_{\sigma(1)}} \cdots (X^2)_{i_r i_{\sigma(r)}} 
      X_{i_{r+1} i_{\sigma(r+1)}} (X^s v)_{i_{r+2}} w_{i_{\sigma(r+2)}},
   $$
   so that the left hand side of the assertion is equal to
   $\sum_{\sigma \in S_{r+2}} \operatorname{sgn}(\sigma) D_{\sigma}$.
   When $\sigma(r+2) = r+2$, we have
   $$
      D_{\sigma}
      = 
      \sum_{1 \leq i_1,\ldots,i_{r+1} \leq n} 
      (X^2)_{i_1 i_{\sigma(1)}} \cdots (X^2)_{i_r i_{\sigma(r)}} 
      X_{i_{r+1} i_{\sigma(r+1)}} \,{}^tw X^s v.
   $$
   Thus we have
   \begin{align*}
      &\sum_{\sigma \in S_{r+2}, \, \sigma(r+2) = r+2} 
      \operatorname{sgn}(\sigma) D_{\sigma} \\
      & \qquad = 
      \sum_{\sigma \in S_{r+2}, \, \sigma(r+2) = r+2} 
      \operatorname{sgn}(\sigma) 
      \sum_{1 \leq i_1,\ldots,i_{r+1} \leq n} 
      (X^2)_{i_1 i_{\sigma(1)}} \cdots (X^2)_{i_r i_{\sigma(r)}} 
      X_{i_{r+1} i_{\sigma(r+1)}} \,{}^tw X^s v \\
      & \qquad = D([X^2]^r, X) \,{}^tw X^s v.
   \end{align*}
   Similarly we have
   \begin{align*}
      &\sum_{\sigma \in S_{r+2}, \, \sigma(r+2) = r+1} 
      \operatorname{sgn}(\sigma) D_{\sigma} \\
      & \qquad =
      \sum_{\sigma \in S_{r+2}, \, \sigma(r+2) = r+1} 
      \operatorname{sgn}(\sigma) 
      \sum_{1 \leq i_1,\ldots,i_{r+1} \leq n} 
      (X^2)_{i_1 i_{\sigma(1)}} \cdots (X^2)_{i_r i_{\sigma(r)}} 
      (X^{s+1} v)_{i_{r+1}} w_{i_{\sigma(r+2)}} \\
      & \qquad = - D([X^2]^r \,|\, X^{s+1}v \,|\, w).
   \end{align*}
   Moreover, for $1 \leq k \leq r$, we have
   \begin{align*}
      &\sum_{\sigma \in S_{r+2}, \, \sigma(r+2) = k} 
      \operatorname{sgn}(\sigma) D_{\sigma} \\
      & \qquad = 
      \sum_{\sigma \in S_{r+2}, \, \sigma(r+2) = k} 
      \operatorname{sgn}(\sigma) 
      \sum_{1 \leq i_1,\ldots,i_{r+1} \leq n} 
      (X^2)_{i_1 i_{\sigma(1)}} \cdots 
      \widehat{(X^2)}_{i_k i_{\sigma(k)}} \cdots (X^2)_{i_r i_{\sigma(r)}} \\ 
      & \qquad\qquad\qquad\qquad\qquad\qquad\qquad\qquad\qquad\qquad\qquad\qquad   
      X_{i_{r+1} i_{\sigma(r+1)}} (X^{s+2} v)_{i_k} w_{i_{\sigma(r+2)}} \\
      &\qquad = - D([X^2]^{r-1}, X \,|\, X^{s+2} v \,|\, w).
   \end{align*}
   Here the hat means that we omit the $k$th factor.
   Combining these, we have the assertion of the last relation.

   We can prove the other two relations similarly looking at the value of $\sigma(r+1)$.
\end{proof}

Using Lemma~\ref{lemma:recurrence_relations1}, we have the following relations:

\begin{lemma}\label{lemma:lemma_for_CH1}\sl
   We have
   \begin{align*}
      D([X^2]^r, X^s) &= (-)^r r!\operatorname{tr}(X^{2r+s}), \\
      D([X^2]^r \,|\, X^s v \,|\, w) &= (-)^r r! \,{}^tw X^{2r+s} v, \\
      D([X^2]^r, X \,|\, X^s v \,|\, w)
      &= (-)^{r+1} \{ 
      (r+1)! \,{}^tw X^{2r+s+1} v 
      - r! \sum_{0 \leq k \leq r} \operatorname{tr}(X^{2k+1}) \,{}^tw X^{2r-2k+s} v 
      \}.
   \end{align*}
\end{lemma}

\begin{proof}
   These three relations can be proved by using three relations in 
   Lemma~\ref{lemma:recurrence_relations1}, respectively. 
\end{proof}

We have Proposition~\ref{prop:equivalent_with_CH1} as a special case of the last relation
in Lemma~\ref{lemma:lemma_for_CH1}.
Thus we have proved Theorem~\ref{thm:CH_type_thm1}.

We have the following relation as a corollary of Theorem~\ref{thm:CH_type_thm1}:

\begin{corollary}\label{cor:X^2n=0}\sl
   We have $X^{2n} = 0$.
\end{corollary}

\begin{proof}
   Multiplying Theorem~\ref{thm:CH_type_thm1} by $X$ from left or right, we have
   \begin{align*}
      n X^{2n} 
      - \operatorname{tr}(X^1) X^{2n-1} 
      - \operatorname{tr}(X^3) X^{2n-3}
      - \cdots 
      - \operatorname{tr}(X^{2n-3}) X^3 
      - \operatorname{tr}(X^{2n-1}) X^1
      &= 0, \\
      n X^{2n} 
      + \operatorname{tr}(X^1) X^{2n-1} 
      + \operatorname{tr}(X^3) X^{2n-3}
      + \cdots 
      + \operatorname{tr}(X^{2n-3}) X^3 
      + \operatorname{tr}(X^{2n-1}) X^1
      &= 0.
   \end{align*}
   Look at the left hand sides of these two relation.
   The first terms are equal,
   but the signs of the other terms are opposite, 
   because $\operatorname{tr}(X^{2k-1})$ are of odd degree.
   Thus, adding these two equalities and dividing by $2n$, we obtain $X^{2n} = 0$.
\end{proof}

Proposition~\ref{prop:strong_FFT1} in the previous section
is now immediate from this corollary.

\begin{remarks}
   (1) The relation $Q = 0$ was deduced from the fact that
   this is an alternating sum of $n+1$ couplings of vectors and covectors.
   Namely, we can regard this as a corollary of 
   the second fundamental theorem of invariant theory for vector invariants \cite{W}.
   Thus, all our results in Sections~\ref{sec:inv_theory1} and \ref{sec:CH_type_thm1} 
   come from the first and second fundamental theorems
   for vector invariants.

   \medskip

   \noindent
   (2) The ordinary Cayley--Hamilton theorem for $A \in \operatorname{Mat}_{n,n}(\mathbb{C})$ 
   can be similarly proved by looking at $D([A]^n \,|\, v \,|\, w)$
   (see \cite{C}).

   \medskip

   \noindent
   (3) A diagrammatic notation due to Penrose (\cite{Pe}; see also \cite{C})
   is useful for the calculations in this section.
   However we need to specify the order multiplications,
   because we work in noncommutative framework. 
   In this paper, we do not use this notation
   because of this trouble caused by the noncommutativity.

   \medskip

   \noindent
   (4) Theorem~\ref{thm:CH_type_thm1} has the lowest degree 
   among monic relations of $X$ whose coefficients are $GL(V)$-invariants.
   This fact follows from Theorem~\ref{thm:FFT_and_SFT1}.

   \medskip

   \noindent
   (5) 
   We can regard $X$ as the most generic matrix among matrices 
   whose entries are anticommuting with each other.
   Thus Theorem~\ref{thm:CH_type_thm1} and Corollary~\ref{cor:X^2n=0} hold for any matrix 
   with anticommuting entries.

   \medskip

   \noindent
   (6) 
   As written in Introduction, 
   Theorem~\ref{thm:CH_type_thm1} was given in \cite{BPS} independently of this article
   (see also \cite{DPP} and \cite{Pr}).   
\end{remarks}

%
\section{Relation to the Amitsur--Levitzki theorem}
\label{sec:AL_thm}
%
%
The Cayley--Hamilton theorem in the previous section is 
closely related to the following Amitsur--Levitzki theorem:

\begin{theorem}[Amitsur--Levitzki \cite{AL}]\label{thm:AL_thm}\sl
   For $2n$ complex square matrices $X_1,\ldots,X_{2n}$ of size $n$, 
   we have
   $$
      \sum_{\sigma \in S_{2n}} 
      \operatorname{sgn}(\sigma) 
      X_{\sigma(1)} X_{\sigma(2)} \cdots X_{\sigma(2n)} = 0.
   $$
\end{theorem}

The original proof given by Amitsur and Levitzki was complicated,
but Rosset \cite{Ros} gave a simple and elementary proof.
The key of this simple proof is the following matrix:
$$
   X = X_1 e_1 + \cdots + X_{2n} e_{2n}.
$$
Here $e_1,\ldots,e_{2n}$ are anticommuting formal variables.
We regard $X$ as an element of $\operatorname{Mat}_{n,n}(\Lambda(\mathbb{C}^n))$,
where $\Lambda(\mathbb{C}^{2n})$ is the exterior algebra generated by $e_1,\ldots,e_{2n}$.
To prove Theorem~\ref{thm:AL_thm},
it suffices to show $X^{2n} = 0$,
because
$$
   X^{2n} = \sum_{\sigma \in S_{2n}} 
   \operatorname{sgn}(\sigma) 
   X_{\sigma(1)} X_{\sigma(2)} \cdots X_{\sigma(2n)}
   e_1 e_2 \cdots e_{2n}.
$$
This relation $X^{2n} = 0$ itself is obtained
by applying the ordinary Cayley--Hamilton theorem to the matrix $X^2$.
Indeed, the entries of $X^2$ are commutative with each other,
so that the Cayley--Hamilton theorem holds for $X^2$,
and the characteristic polynomial of $X^2$ is equal to $\lambda^n$
because $\operatorname{tr}(X^2) = \operatorname{tr}(X^4) = \cdots = 0$.
This is the proof of Theorem~\ref{thm:AL_thm} given by Rosset \cite{Ros}.

Actually, we have proved the key relation $X^{2n} = 0$ in
Corollary~\ref{cor:X^2n=0} as a corollary of Theorem~\ref{thm:CH_type_thm1}.
Indeed, the entries of $X$ are anticommuting with each other.
This tells us that we can regard our Cayley--Hamilton type theorem 
as a refinement of the Amitsur--Levitzki theorem.

\begin{remark}
   Various proofs of the Amitsur--Levitzki theorem had been known also before the proof by Rosset. 
   Kostant proved this using Lie algebra cohomology \cite{K1}
   (this method also gave another Amitsur--Levitzki type theorem (Theorem~\ref{thm:Kostant})).
   Swan proved this using graph theory \cite{S}.
   Razmyslov proved this as a consequence of the Cayley--Hamilton theorem \cite{Ra}
   (actually, he proved that all trace identities are 
   a consequence of the Cayley--Hamilton theorem).
\end{remark}

%
\section{Amitsur--Levitzki type theorem due to Kostant and Rowen}
\label{sec:AL_type_thm_due_to_Kostant}
%
%
In Sections~\ref{sec:AL_type_thm_due_to_Kostant} and \ref{sec:new_AL_type_thm}, 
we will discuss two more examples of Amitsur--Levitzki type theorems
related to invariant theory in exterior algebras.

First, in this section,
we investigate the following famous Amitsur--Levitzki type theorem.
We will see that this is related to $O(V)$-invariants in the exterior algebra $\Lambda(\Lambda_2(V))$.

\begin{theorem}[Kostant \cite{K1}, Rowen \cite{Row1}] \label{thm:Kostant}\sl
   For $2n-2$ complex alternating matrices $A_1,\ldots,A_{2n-2}$ of size $n$,
   we have
   $$
      \sum_{\sigma \in S_{2n-2}} 
      \operatorname{sgn}(\sigma) 
      A_{\sigma(1)}  A_{\sigma(2)} \cdots A_{\sigma(2n-2)} = 0.
   $$
\end{theorem}

This theorem is much more difficult to prove than Theorem~\ref{thm:AL_thm}
(the method used in the proof of Theorem~\ref{thm:AL_thm} due to Rosset is not valid).
Kostant first proved Theorem~\ref{thm:Kostant} using theory of cohomology of Lie algebras,
when $n$ is even (\cite{K1}; see also \cite{K2}).
Later, Rowen gave more elementary but technical proof 
for arbitrary $n$ (\cite{Row1}; see also \cite{Row2}).
In this section, we will connect this Theorem~\ref{thm:Kostant}
with invariant theory for $O(V)$-invariants in the exterior algebra $\Lambda(\Lambda_2(V))$,
and deduce this from a Cayley--Hamilton type theorem
similar to Theorem~\ref{thm:CH_type_thm1}.

\subsection{Invariant theory for {\boldmath$O(V)$}-invariants in {\boldmath$\Lambda(\Lambda_2(V))$}}
\label{subsec:inv_theory2}
We study invariants under the natural action of the orthogonal group $O(V)$ 
in the exterior algebra $\Lambda(\Lambda_2(V))$ 
on the second antisymmetric tensor $\Lambda_2(V)$ of $V$.
Here $V$ is an $n$-dimensional complex vector space
with a symmetric bilinear form.
We fix an orthonormal basis $e_1,\ldots,e_n$ of $V$ 
and put $a_{ij} = e_i \wedge e_j$ in $\Lambda_2(V)$.
Then, for the matrix
$A = (a_{ij})_{1 \leq i,j \leq n} \in \operatorname{Mat}_{n,n}(\Lambda(\Lambda_2(V)))$,
we have the following propositions:

\begin{proposition}\sl
   $\operatorname{tr}(A^k)$ is $O(V)$-invariant.
\end{proposition}

\begin{proof}
   This is immediate from the equality
   \begin{equation}\label{eq:pi(g)A}
      \pi(g)A 
      = (\pi(g)a_{ij})_{1 \leq i,j \leq n} 
      = {}^t\!g A \, {}^t\!g^{-1}
      = {}^t\!g A g.
   \end{equation}
   Here we denote by $\pi$ the natural action of $O(V)$ on $\Lambda(\Lambda_2(V))$.
\end{proof}

\begin{proposition}\label{prop:A^l} \sl
   The matrix $A^l$ is symmetric when $l \equiv 0,3 \mod 4$,
   and is alternating when $l \equiv 1,2 \mod 4$.
\end{proposition}

\begin{proof}
   For example, $A^3$ is symmetric, because
   \begin{align*}
      (A^3)_{ij}
      &= \sum_{1 \leq r,s \leq n} a_{ir} a_{rs} a_{sj}
      = \sum_{1 \leq r,s \leq n} (-a_{ri}) (-a_{sr}) (-a_{js}) 
      \allowdisplaybreaks\\
      &= -\sum_{1 \leq r,s \leq n} a_{ri} a_{sr} a_{js}
      = \sum_{1 \leq r,s \leq n} a_{js} a_{sr} a_{ri}
      = (A^3)_{ji}.
   \end{align*}
   The other cases are similarly shown.
\end{proof}

\begin{proposition}\label{prop:tr(A^l)} \sl
   For $l \geq 1$, we have $\operatorname{tr}(A^l) = 0$ unless $l \equiv 3 \mod 4$.
\end{proposition}

\begin{proof}
   We see that $\operatorname{tr}(A^2) = \operatorname{tr}(A^4) = \cdots = 0$
   in a way similar to the proof of Theorem~\ref{thm:tr(X^2r)=0}.
   The assertion is immediate from this and Proposition~\ref{prop:A^l}.
\end{proof}

Let us put $q_{4m+3} = \operatorname{tr}(A^{4m+3})$. 
Then we have the following theorem:

\begin{theorem}\label{thm:FFT_and_SFT2} \sl
   The algebra $\Lambda(\Lambda_2(V))^{O(V)}$ is generated by the following elements,
   and these are anticommuting with each other and have no other relations{\rm :}
   \begin{align*}
      q_3, q_7, q_{11}, \ldots, q_{4m-5}, \qquad n&=2m, \\
      q_3, q_7, q_{11}, \ldots, q_{4m-1}, \qquad n&=2m+1.
   \end{align*}
\end{theorem}

We will prove this as the composition of the following three propositions:

\begin{proposition}\label{prop:weak_FFT2}\sl
   $\Lambda(\Lambda_2(V))^{O(V)}$ is generated by $q_3, q_7, q_{11},\ldots$.
\end{proposition}

\begin{proposition}\label{prop:strong_FFT2}\sl
   For $k \geq 2n-2$, we have $\operatorname{tr}(A^k) = 0$.
\end{proposition}

\begin{proposition}\label{prop:SFT2}\sl
   The following elements are anticommuting with each other and have no other relations{\rm :}
   \begin{align*}
      q_3, q_7, q_{11}, \ldots, q_{4m-5}, \qquad n&=2m, \\
      q_3, q_7, q_{11}, \ldots, q_{4m-1}, \qquad n&=2m+1.
   \end{align*}
\end{proposition}

\begin{proof}[Proof of Proposition~{\sl\ref{prop:weak_FFT2}}]
   The homogeneous decomposition 
   $$
      \Lambda(\Lambda_2(V)) = \bigoplus_{k=0}^{n(n-1)/2}\Lambda_k(\Lambda_2(V))
   $$
   is a decomposition as $O(V)$-spaces,
   so that it suffices to describe the $O(V)$-invariants in $\Lambda_k(\Lambda_2(V))$.
   The following map $V^{\otimes 2k} \to \Lambda_k(\Lambda_2(V))$
   is a surjective homomorphism of $O(V)$-spaces:
   $$
      e_{i_1} \otimes e_{j_1} \otimes \cdots \otimes e_{i_k} \otimes e_{j_k}
      \mapsto
      a_{i_1j_1} \cdots a_{i_kj_k}.
   $$
   Thus any $O(V)$-invariant in $\Lambda_k(\Lambda_2(V))$
   comes from an $O(V)$-invariants in $V^{\otimes 2k}$.
   By the first fundamental theorem of invariant theory for vector invariants (\cite{W}, \cite{GW}), 
   any $O(V)$-invariant in $V^{\otimes 2k}$
   can be expressed as a linear combination of elements in the form
   $$
      \sum_{(i_1,\ldots,i_{2k}) \in I}
      e_{i_{\sigma(1)}} \otimes e_{i_{\sigma(2)}} \otimes \cdots 
      \otimes e_{i_{\sigma(2k-1)}} \otimes e_{i_{\sigma(2k)}}
   $$
   with $\sigma \in S_{2k}$.
   Here we put
   \begin{align*}
      I 
      &= \{ 
      (i_1,\ldots,i_{2k}) \in \{1,\ldots,n\}^{2k} 
      \,|\,
      i_1 = i_2, \, i_3 = i_4, \, \ldots, \, i_{2k-1} = i_{2k}    
      \} \\
      &= \{ (j_1,j_1,j_2,j_2,\ldots,j_k,j_k) \,|\, j_1,\ldots,j_k \in \{1,\ldots,n\} \}.
   \end{align*}
   The image of this element is equal to 
   $$
      \sum_{(i_1,\ldots,i_{2k}) \in I}
      a_{i_{\sigma(1)} i_{\sigma(2)}} \cdots a_{i_{\sigma(2k-1)} i_{\sigma(2k)}},
   $$
   and this is equal to a product of $q_3,q_7,q_{11},\ldots$ up to a sign.
   Thus any $O(V)$-invariant in $\Lambda_k(\Lambda_2(V))$ is expressed as a linear combination 
   of products of $q_3,q_7,q_{11},\ldots$.
\end{proof}

Proposition~\ref{prop:strong_FFT2} is seen from a Cayley--Hamilton type theorem 
in the next subsection.

To prove Proposition~\ref{prop:SFT2},
we consider the following product of $n(n-1)/2$ elements:
\begin{align*}
   h = a_{12} a_{13} \cdots &\,a_{1n}  \\
   a_{23} \cdots &\,a_{2n}  \\
   &\,\,\vdots \\
   &\,a_{n-1,n}. 
\end{align*}

\begin{lemma}\sl
   We have $\pi(g) h = \det(g) h$ for $g \in O(V)$.
   Here $\pi$ is the natural action of $GL(V)$ on $\Lambda(\Lambda_2(V))$.
   Thus $h$ is $O(V)$-invariant when $n$ is odd.
   However, when $n$ is even, this is not $O(V)$-invariant, 
   but $SO(V)$-invariant.
\end{lemma}

\begin{proof}
   We have $\det \rho(g) = (\det g)^{n-1}$
   for the natural action $\rho$ of $GL(V)$ on $\Lambda_2(V)$.
   Thus, by Lemma~\ref{lemma:invariance_of_the_highest_element},
   we have
   $$
      \pi(g) h 
      = \det(\rho(g))^{n(n-1)/2} h
      = \det(g)^{n(n-1)^2/2} h
   $$
   for $g \in GL(V)$.
   The assertion is immediate from this.
\end{proof}

Let us assume that $n = 2m+1$.
Then $h$ is $O(V)$-invariant,
so that this is generated by $q_3, q_7, q_{11},\ldots,q_{4m-1}$.
Namely $h$ is equal to the product $q_3 q_7 q_{11} \cdots q_{4m-1}$ up to constant,
because
$$
   3 + 7 + 11 + \cdots + (4m-1) = \frac{n(n-1)}{2}.
$$
Thus we have $q_3 q_7 q_{11} \cdots q_{4m-1} \ne 0$.
This means that $q_3, q_7, q_{11},\ldots,q_{4m-1}$ have no relations 
besides the anticommutativity.

Next we assume that $n = 2m$.
We consider the following Pfaffian type element:
$$
   p
   = \sum_{\sigma \in S_{2m}} \operatorname{sgn}(\sigma)
   (A^2)_{\sigma(1)\sigma(2)} \cdots (A^2)_{\sigma(2m-3)\sigma(2m-2)} A_{\sigma(2m-1)\sigma(2m)}.
$$
The coefficient of $a_{1n} a_{2n} a_{3n} a_{n-1, n}$ in $p$ is equal to $(-)^{m-1}2^m m!(2m-1)$,
so that $p \ne 0$.

\begin{lemma}\sl
   We have $\pi(g) p = \det(g) p$ for $g \in O(V)$.
\end{lemma}

\begin{proof}
   For a $\mathbb{C}$-algebra $R$ and $X_1,\ldots,X_m \in \operatorname{Mat}_{2m,2m}(R)$, we put
   $$
      \operatorname{pf}(X_1,\ldots,X_m) 
      = \sum_{\sigma \in S_{2m}} \operatorname{sgn}(\sigma)
      (X_1)_{\sigma(1)\sigma(2)} (X_2)_{\sigma(3)\sigma(4)} \cdots (X_m)_{\sigma(2m-1)\sigma(2m)}.
   $$
   Then we have
   $$
      \operatorname{pf}(\,{}^t\!g X_1 g,\ldots, \,{}^t\!g X_m g) 
      = \det(g)\operatorname{pf}(X_1,\ldots,X_m)
   $$
   for $g \in GL_{2m}(\mathbb{C})$.
   The assertion is immediate from this general fact and (\ref{eq:pi(g)A}).
\end{proof}

We consider the natural bilinear form $\langle \cdot \,|\, \cdot \rangle$ on $\Lambda_2(V)$
determined by the bilinear form on $V$ ($a_{ij}$ forms an orthonormal basis),
and consider the operator $\operatorname{Der}(p) \in \operatorname{End}(\Lambda(\Lambda_2(V)))$.

Here ``$\operatorname{Der}$'' is defined as follows.
Let $W$ be a complex vector space with a nondegenerate symmetric bilinear form
$\langle \cdot \,|\, \cdot \rangle$,
and consider an action $\pi$ of a group $G$ on $W$ preserving $\langle \cdot \,|\, \cdot \rangle$
(namely we fix a homomorphism $G \to O(W)$).
For $a \in W$,
we define the derivation $\operatorname{Der}(a) \in \operatorname{End}(\Lambda(W))$ by
$$
   \operatorname{Der}(a) \colon b_1 \cdots b_k 
   \mapsto \sum_{i=1}^k (-)^{i-1} \langle a \,|\, b_i \rangle 
   b_1 \cdots \hat{b}_i \cdots b_k,
$$
where $b_1,\ldots,b_k$ are elements of $W$.
Moreover we extend this as an algebra homomorphism 
$\operatorname{Der} \colon \Lambda(W) \to \operatorname{End}(\Lambda(W))$.
Then we have $\operatorname{Der}(\pi(g)(x)) \pi(g)(y) = \pi(g)(\operatorname{Der}(x)y)$ 
for $g \in G$ and $x$, $y \in \Lambda(W)$,
where $\pi$ is the natural action of $G$ on $\Lambda(W)$.

\begin{lemma}\sl
   $\operatorname{Der}(p)h$ is a nonzero $O(V)$-invariant of degree $n(n-1)/2 - (n-1)$.
\end{lemma}

Thus, $\operatorname{Der}(p)h$ is generated by $q_3, q_7, q_{11},\ldots,q_{4m-5}$.
Moreover, since
$$
   3 + 7 + 11 + \cdots + (4m-5) = \frac{n(n-1)}{2} - (n-1),
$$
we see that $\operatorname{Der}(p)h$ is equal to $q_3 q_7 q_{11} \cdots q_{4m-5}$ up to constant,
so that $q_3 q_7 q_{11} \cdots q_{4m-5} \ne 0$.
This means that $q_3, q_7, q_{11},\ldots,q_{4m-5}$ have no relations 
besides the anticommutativity.

\subsection{Cayley--Hamilton type theorem for {\boldmath$\Lambda(\Lambda_2(V))$}}
\label{subsec:CH_type_thm2}
We have the following Cayley--Hamilton type theorem%
\footnote{As written in Introduction,
the case $n=2m+1$ of this theorem
is also given by Dolce \cite{D} independently of this article.}:

\begin{theorem}\label{thm:CH_type_thm2} \sl
   We have the following relation in $\operatorname{Mat}_{n,n}(\Lambda(\Lambda_2(V)))${\rm :}
   \begin{align*}
      (n-2)A^{2n-3} - \sum_{0 \leq k \leq m-2} \operatorname{tr}(A^{4k+3}) A^{2n-3-4k-3} &= 0,
      \qquad
      n=2m, \\
      nA^{2n-3} - \sum_{0 \leq k \leq m-1} \operatorname{tr}(A^{4k+3}) A^{2n-3-4k-3} &= 0,
      \qquad
      n=2m+1.
   \end{align*}
\end{theorem}

As a consequence of this theorem, we have the following relation. 
The proof is almost the same as that of Corollary~\ref{cor:X^2n=0}.

\begin{corollary}\label{cor:A^2n-2=0}\sl
   We have $A^{2n-2} = 0$.
\end{corollary}

Proposition~\ref{prop:strong_FFT2} is immediate from this corollary.

Theorem~\ref{thm:Kostant} also follows from this corollary
in a way similar to the discussion in Section~\ref{sec:AL_thm}.
Thus, we can regard Theorem~\ref{thm:CH_type_thm2}
as a refinement of this Amitsur--Levitzki type theorem.

The proof of Theorem~\ref{thm:CH_type_thm2} is similar to that of Theorem~\ref{thm:CH_type_thm1}.
The calculation is harder,
but this also follows from the second fundamental theorem for vector invariants. 
In addition to (\ref{eq:def_of_D1}) and (\ref{eq:def_of_D2}), we put
\begin{align*}
   &D(\Omega_1,\ldots,\Omega_r \,|\, \Phi \,|\, \alpha, \beta) \\
   &\qquad
   = \sum_{\sigma \in S_{r+2}} \sum_{1 \leq i_1,\ldots,i_{r+2} \leq n} 
   \operatorname{sgn}(\sigma)
   \Omega_{i_1 i_{\sigma(1)}} \cdots 
   \Omega_{i_r i_{\sigma(r)}} 
   \Psi_{i_{r+1} i_{r+2}}
   \alpha_{i_{\sigma(r+1)}}
   \beta_{i_{\sigma(r+2)}},
   \allowdisplaybreaks\\
   &D(\Omega_1,\ldots,\Omega_r \,|\, \Phi \,|\, \Psi) \\
   &\qquad
   = \sum_{\sigma \in S_{r+3}} \sum_{1 \leq i_1,\ldots,i_{r+2} \leq n} 
   \operatorname{sgn}(\sigma)
   \Omega_{i_1 i_{\sigma(1)}} \cdots \Omega_{i_r i_{\sigma(r)}} 
   \Phi_{i_{r+1} i_{r+2}}
   \Psi_{i_{\sigma(r+1)} i_{\sigma(r+2)}}, 
   \allowdisplaybreaks\\   
   &D(\Omega_1,\ldots,\Omega_r \,|\, \Phi,\alpha \,|\, \Psi, \beta) \\
   &\qquad
   = \sum_{\sigma \in S_{r+3}} \sum_{1 \leq i_1,\ldots,i_{r+3} \leq n} 
   \operatorname{sgn}(\sigma)
   \Omega_{i_1 i_{\sigma(1)}} \cdots \Omega_{i_r i_{\sigma(r)}} 
   \Phi_{i_{r+1} i_{r+2}}
   \alpha_{i_{r+3}}
   \Psi_{i_{\sigma(r+1)} i_{\sigma(r+2)}}
   \beta_{i_{\sigma(r+3)}}
\end{align*}
for $\Omega_1,\ldots,\Omega_r,\Phi,\Psi \in \operatorname{Mat}_{n,n}(R)$
and $\alpha$, $\beta \in \operatorname{Mat}_{n,1}(R)$.
It is easily seen that 
$$D(\Omega_1,\ldots,\Omega_r \,|\, \Phi \,|\, \alpha, \beta) = 0,$$ 
when $\Phi$ is a symmetric matrix.
Moreover we have  
$$
   D(\Omega_1,\ldots,\Omega_r \,|\, \Phi \,|\, \Psi) = 0, \qquad
   D(\Omega_1,\ldots,\Omega_r \,|\, \Phi,\alpha \,|\, \Psi, \beta) = 0,
$$
when $\Phi$ or $\Psi$ is a symmetric matrix.

The following element is the key of the proof:
$$
   Q = D([A^2]^{n-3}, A \,|\, A, v \,|\, A,w).
$$
Here $v_1,\ldots,v_n,w_1,\ldots,w_n$ are arbitrary complex numbers, and
we put $v = {}^t(v_1,\ldots,v_n)$ and $w = {}^t(w_1,\ldots,w_n)$.

We have $Q = 0$ in a way similar to the discussion in Section~\ref{sec:CH_type_thm1},
so that Theorem~\ref{thm:CH_type_thm2} is immediate from the following relation:

\begin{proposition}\label{prop:equivalent_with_CH2}\sl 
   When $n=2m$, we have
   \begin{align*}
      &D([A^2]^{n-3},A \,|\, A, v \,|\, A, w) \\
      & \qquad
      = - 2n (n-3)! 
         \{ (n-2) \,{}^tw A^{2n-3} v 
         - \sum_{0 \leq k \leq m-2} \operatorname{tr}(A^{4k+3}) \,{}^tw A^{2n-3 - 4k-3}v \}.
   \end{align*}
   When $n=2m+1$, we have
   \begin{align*}
      &D([A^2]^{n-3},A \,|\, A, v \,|\, A, w) \\
      & \qquad
      = - 2(n-1)(n-3)! 
         \{ n \,{}^tw A^{2n-3} v 
         - \sum_{0 \leq k \leq m-1} \operatorname{tr}(A^{4k+3}) \,{}^tw A^{2n-3 - 4k-3} v \}. 
   \end{align*}
\end{proposition}

Let us prove this.
First we note the following recurrence relations
(Laplace type expansions).
The proof is almost the same as that of Lemma~\ref{lemma:recurrence_relations1}:

\begin{lemma}\label{lemma:recurrence_relations2}\sl
   We have 
   \begin{align}
      \label{eq:D([A^2]^r,A^s)}
      D([A^2]^r, A^s) &= D([A^2]^r) \operatorname{tr}(X^s) - r D([A^2]^{r-1}, A^{s+2}),
      \allowdisplaybreaks\\
      \label{eq:D([A^2]^r|A|A^s)}
      D([A^2]^r \,|\, A \,|\, A^s) &= 
      -2 D([A^2]^r,A^{s+1}) 
      + r D([A^2]^{r-1} \,|\, A \,|\, A^{s+2}), 
      \allowdisplaybreaks\\
      \label{eq:D([A^2]^r,A|A|A)}
      D([A^2]^r, A \,|\, A \,|\, A) 
      &= -2 D([A^2]^r, A, A^2) \\
      & \qquad 
      - D([A^2]^r \,|\, A \,|\, A^2)
      + r D([A^2]^{r-1}, A \,|\, A \,|\, A^3).
   \notag
   \end{align}
\end{lemma}

Using this and Propositions~\ref{prop:A^l} and \ref{prop:tr(A^l)}, we have the following relations:

\begin{lemma}\sl
   We have
   \begin{alignat*}{2}
      D([A^2]^r) &= 0, 
      \qquad &
      D([A^2]^r, A) &= (-)^r r!\operatorname{tr}(A^{2r+1}), \\
      D([A^2]^r \,|\, A \,|\, A) &= 0, 
      \qquad &
      D([A^2]^r, A \,|\, A \,|\, A) &= (-)^r 2(r+2)r! \operatorname{tr}(A^{2r+3}).
   \end{alignat*}
\end{lemma}

\begin{proof}
   We obtain the first and second relations 
   using (\ref{eq:D([A^2]^r,A^s)}) repeatedly.
   
   Let us prove the third relation.
   By (\ref{eq:D([A^2]^r|A|A^s)}), we have
   $$
      D([A^2]^r \,|\, A \,|\, A)
      = -2 D([A^2]^r,A^2) 
         + r D([A^2]^{r-1} \,|\, A \,|\, A^3).
   $$
   Note that $D([A^2]^r,A^2) = D([A^2]^{r+1}) = 0$.
   Moreover we have $D([A^2]^{r-1} \,|\, A \,|\, A^3) = 0$, because $A^3$ is symmetric. 
   Thus we have the third relation.

   To prove the fourth relation,
   we look at the right hand side of (\ref{eq:D([A^2]^r,A|A|A)}).
   We can compute the first term,
   because $D([A^2]^r, A, A^2) = D([A^2]^{r+1}, A) = (-)^{r+1}(r+1)! \operatorname{tr}(A^{2r+3})$.
   Next, the last term is equal to $0$, because $A^3$ is symmetric.
   We can also compute the second term.
   Indeed, by (\ref{eq:D([A^2]^r|A|A^s)}), we have 
   $$
      D([A^2]^r \,|\, A \,|\, A^2)
      = -2 D([A^2]^r,A^3) 
         + r D([A^2]^{r-1} \,|\, A \,|\, A^4).
   $$
   Here we can compute $D([A^2]^r,A^3)$ using (\ref{eq:D([A^2]^r,A^s)}) repeatedly.
   Moreover $D([A^2]^{r-1} \,|\, A \,|\, A^4)$ is equal to $0$, because  $A^4$ is symmetric.
   Combining these, we have the fourth relation.
\end{proof}

Moreover we have the following recurrence relations (Laplace type expansions):

\begin{lemma}\label{lemma:recurrence_relations3}\sl
   We have 
   \begin{align*}
      D([A^2]^r \,|\, A^s v \,|\, \,w)
      &= D([A^2]^r) \,{}^tw A^s v
      - r D([A^2]^{r-1} \,|\, A^{s+2} v \,|\, w), \\
      D([A^2]^r, A \,|\, A^s v \,|\, w)
      &= D([A^2]^r, A) \,{}^tw A^s v \\
      &\qquad 
      - D([A^2]^r \,|\, A^{s+1} v \,|\, w)
      - r D([A^2]^{r-1}, A \,|\, A^{s+2} v \,|\, w), 
      \allowdisplaybreaks\\ 
      D([A^2]^r \,|\, A \,|\, w, A^s v) 
      &= 2 D([A^2]^r \,|\, A^{s+1} v \,|\, w)
      + r D([A^2]^{r-1} \,|\, A \,|\, w, A^{s+2} v), 
      \allowdisplaybreaks\\
      D([A^2]^r, A \,|\, A \,|\, w, A^s v) 
      &= 
      2 D([A^2]^r,A \,|\, A^{s+1} v \,|\, w)
      - D([A^2]^r \,|\, A \,|\, w, A^{s+1} v) \\
      & \qquad      
      + r D([A^2]^{r-1}, A \,|\, A \,|\, w, A^{s+2} v),
      \allowdisplaybreaks\\
      D([A^2]^r \,|\, A, A^s v \,|\, A, w) 
      &= 
      (-)^s D([A^2]^r \,|\, A \,|\, A) \,{}^tw A^s v
      + (-)^s 2 D([A^2]^r \,|\, A \,|\, w, A^{s+1} v) \\
      & \qquad
      -r D([A^2]^{r-1} \,|\, A, A^{s+2} v \,|\, A, w),
      \allowdisplaybreaks\\
      D([A^2]^r, A \,|\, A, A^s v \,|\, A, w) &= 
      (-)^s D([A^2]^r, A \,|\, A \,|\, A) \,{}^tw A^s v \\
      & \qquad
      + (-)^s 2 D([A^2]^r, A \,|\, A \,|\, w, A^{s+1} v) \\
      & \qquad
      + D([A^2]^r \,|\, A, A^{s+1} v \,|\, A, w) \\
      & \qquad
      - r D([A^2]^{r-1}, A \,|\, A, A^{s+2} v \,|\, A, w) 
   \end{align*}
\end{lemma}

Using this we have the following relations by induction on $r$:

\begin{lemma}\label{lemma:lemma_for_CH2} \sl
   When $r = 2l$, we have
   \begin{align*}
      D([A^2]^r \,|\, A^s v \,|\, \,w)
      &= r! \,{}^tw A^{2r+s} v, 
      \allowdisplaybreaks\\
      D([A^2]^r, A \,|\, A^s v \,|\, w)
      &= -
      (r+1)! \,{}^tw A^{2r+s+1} v
      + r! \sum_{0 \leq k \leq l-1} \operatorname{tr}(A^{4k+3})
      \,{}^tw A^{2r+s+1 -4k-3} v, \\
      D([A^2]^r \,|\, A \,|\, w, A^s v) 
      &= 2 r! \,{}^tw A^{2r+s+1} v, 
      \allowdisplaybreaks\\
      D([A^2]^r, A \,|\, A \,|\, w, A^s v) 
      &= -2(r+2)r! \,{}^tw A^{2r+s+2} v, 
      \allowdisplaybreaks\\
      D([A^2]^r \,|\, A, A^s v \,|\, A, w) 
      &= (-)^s 2(r+2) r! \,{}^tw A^{2r+s+2} v, 
      \allowdisplaybreaks\\
      D([A^2]^r, A \,|\, A, A^s v \,|\, A, w) 
      &= (-)^{s+1} 2(r+2)r! 
         \{ (r+3) \,{}^tw A^{2r+s+3} v \\
         &\qquad
         - \sum_{0 \leq k \leq l} \operatorname{tr}(A^{4k+3}) \,{}^tw A^{2r+s+3 - 4k-3} v \}.
   \end{align*}
   When $r = 2l+1$, we have
   \begin{align*}
      D([A^2]^r \,|\, A^s v \,|\, \,w)
      &= - r! \,{}^tw A^{2r+s} v, 
      \allowdisplaybreaks\\
      D([A^2]^r, A \,|\, A^s v \,|\, w)
      &= 
      (r+1)! \,{}^tw A^{2r+s+1} v
      - r! \sum_{0 \leq k \leq l} \operatorname{tr}(A^{4k+3})
      \,{}^tw A^{2r+s+1 -4k-3} v, \\
      D([A^2]^r \,|\, A \,|\, w, A^s v) 
      &= 0, 
      \allowdisplaybreaks\\
      D([A^2]^r, A \,|\, A \,|\, w, A^s v) 
      &= -2r! \sum_{0 \leq k \leq l} \operatorname{tr}(A^{4k+3}) \,{}^tw A^{2r+s+2 - 4k-3} v, 
      \allowdisplaybreaks\\
      D([A^2]^r \,|\, A,A^s v \,|\, A, w) 
      &= (-)^{s+1} 2 (r+1)! \,{}^tw A^{2r+s+2} v, 
      \allowdisplaybreaks\\
      D([A^2]^r, A \,|\, A, A^s v \,|\, A, w) 
      &= (-)^s 2(r+3)r! 
         \{ (r+1) \,{}^tw A^{2r+s+3} v \\
         &\qquad
         - \sum_{0 \leq k \leq l} \operatorname{tr}(A^{4k+3}) \,{}^tw A^{2r+s+3 - 4k-3}v \}.
   \end{align*}
\end{lemma}

Proposition~\ref{prop:equivalent_with_CH2} is immediate
from the last relation in this lemma.
Thus we have proved Theorem~\ref{thm:CH_type_thm2}.

\begin{remarks}
   (1) Theorem~\ref{thm:CH_type_thm2} has the lowest degree 
   among monic relations of $A$ whose coefficients are $O(V)$-invariants.
   This fact follows from Theorem~\ref{thm:FFT_and_SFT2}.

   \medskip

   \noindent
   (2) We can also prove Corollary~\ref{cor:A^2n-2=0} directly 
   not using Theorem~\ref{thm:CH_type_thm2}.
   This direct proof is easier than the proof through Theorem~\ref{thm:CH_type_thm2}. 
   Indeed, we only have to show the following:
   $$
      D([A^2]^{n-2} \,|\, A,v \,|\, A,w) = 
      \begin{cases}
         2n(n-2)! {}^tw A^{2n-2} v, & \text{$n$: even}, \\
         -2(n-1)! {}^tw A^{2n-2} v, & \text{$n$: odd}.
      \end{cases}
   $$
   To show this, we only need (\ref{eq:D([A^2]^r,A^s)}) and (\ref{eq:D([A^2]^r|A|A^s)}) 
   in Lemma~\ref{lemma:recurrence_relations2}
   and the first, third and fifth relations in Lemma~\ref{lemma:recurrence_relations3}
   among the recurrence relations used in the proof of Theorem~\ref{thm:CH_type_thm2}. 
\end{remarks}

%
\section{Amitsur--Levitzki type theorem \\ for alternating and symmetric matrices}
\label{sec:new_AL_type_thm}
%
%
Finally, in this section, we give a new Amitsur--Levitzki type theorem:

\begin{theorem} \label{thm:new_AL_type_thm}\sl
   For $n$ complex alternating matrices $A_1,\ldots,A_n$ 
   and $n-1$ complex symmetric matrices $B_1,\ldots,B_{n-1}$ of size $n$,
   we have 
   $$
      \sum_{\sigma \in S_n, \,\, \tau \in S_{n-1}} 
      \operatorname{sgn}(\sigma) \operatorname{sgn}(\tau)
      A_{\sigma(1)} B_{\tau(1)} A_{\sigma(2)} B_{\tau(2)} \cdots 
      A_{\sigma(n-1)} B_{\tau(n-1)} A_{\sigma(n)} = 0.
   $$
\end{theorem}

This can be regarded as a refined version of the following relation:

\begin{theorem}[Giambruno \cite{G}]\label{thm:Giambruno} \sl
   For $n$ complex alternating matrices $A_1,\ldots,A_n$ 
   and $n$ complex symmetric matrices $B_1,\ldots,B_n$ of size $n$,
   we have 
   $$
      \sum_{\sigma, \tau \in S_n} 
      \operatorname{sgn}(\sigma) \operatorname{sgn}(\tau)
      A_{\sigma(1)} B_{\tau(1)} A_{\sigma(2)} B_{\tau(2)} \cdots 
      A_{\sigma(n-1)} B_{\tau(n-1)} A_{\sigma(n)} B_{\tau(n)}= 0.
   $$
\end{theorem}

The proof of Theorem~\ref{thm:Giambruno} is easy
(the method due to Rosset stated in Section~\ref{sec:AL_thm} is valid).
However the proof of Theorem~\ref{thm:new_AL_type_thm} is much more difficult.
This theorem is also related to invariant theory for an exterior algebra
and proved through this relationship.

\subsection{Invariant theory for {\boldmath$GL(V)$}-invariants in {\boldmath$\Lambda(\Lambda_2(V) \oplus S_2(V^*))$}}
\label{subsec:inv_theory3}
Theorem~\ref{thm:new_AL_type_thm} is related to
$GL(V)$-invariants in the exterior algebra $\Lambda(\Lambda_2(V) \oplus S_2(V^*))$
on the direct product of the second antisymmetric tensor $\Lambda_2(V)$ of $V$ and 
the second symmetric tensor $S_2(V^*)$ of $V^*$. 

We do not have nontrivial $GL(V)$-invariants in this exterior algebra:

\begin{theorem}\label{thm:FFT_and_SFT3}\sl
   We have $\Lambda(\Lambda_2(V) \oplus S_2(V^*))^{GL(V)} = \mathbb{C}1$.
\end{theorem}

To prove this,
we consider the standard bases $a_{ij}$ and $b_{ij}$ of $\Lambda_2(V)$ and $S_2(V^*)$,
respectively.
Namely we put 
$$
   a_{ij} = e_i \otimes e_j - e_j \otimes e_i, \qquad
   b_{ij} = e_i \otimes e_j + e_j \otimes e_i,
$$
where $e_i$ is a basis of $V$, and $e^*_i$ is its dual basis.
Moreover we consider the matrices 
$$
   A=(a_{ij})_{1 \leq i,j \leq n}, \qquad
   B=(b_{ij})_{1 \leq i,j \leq n}
$$
in $\operatorname{Mat}_{n,n}(\Lambda(\Lambda_2(V) \oplus S_2(V^*)))$.
For these matrices, we have the following relation:

\begin{proposition}\label{prop:tr(AB)^l}\sl
   We have $\operatorname{tr}(AB)^k = 0$ for any $k > 0$.
\end{proposition}

\begin{proof}
   First, we note ${}^t(AB) = - {}^t\!B \, {}^t\!A = BA$.
   Since the entries of $AB$ are commutative with each other, we have
   $$
      {}^t((AB)^k) = ({}^t(AB))^k = (BA)^k.
   $$
   Thus we have
   \begin{align*}
      \operatorname{tr}(AB)^k 
      &= \operatorname{tr}{}^t((AB)^k) 
      = \operatorname{tr}(BA)^k \\
      &= \sum_{1 \leq i_1,\ldots,i_{2k} \leq n} 
      b_{i_1 i_2} a_{i_2 i_3} b_{i_3 i_4} a_{i_4 i_5} \cdots b_{i_{2k-1} i_{2k}} a_{i_{2k} i_1} \\
      &= -\sum_{1 \leq i_1,\ldots,i_{2k} \leq n} 
      a_{i_2 i_3} b_{i_3 i_4} a_{i_4 i_5} \cdots b_{i_{2k-1} i_{2k}} a_{i_{2k} i_1} b_{i_1 i_2}
      = - \operatorname{tr}(AB)^k.
   \end{align*}
   The assertion is immediate from this.
\end{proof}

\begin{proof}[Proof of Theorem~{\sl\ref{thm:FFT_and_SFT3}}]
   We consider the homogeneous decomposition 
   $$
      \Lambda(\Lambda_2(V) \oplus S_2(V^*)) 
      = \bigoplus_{r=0}^{n(n-1)/2} \bigoplus_{s=0}^{n(n+1)/2}
      \Lambda_r(\Lambda_2(V)) \otimes \Lambda_s(S_2(V^*)).
   $$
   This is a decomposition as $GL(V)$-spaces,
   so that we only have to describe the $GL(V)$-invariants 
   in $\Lambda_r(\Lambda_2(V)) \otimes \Lambda_s(S_2(V^*))$.
   The following map is a surjective homomorphism of $GL(V)$-spaces:
   \begin{gather*}
      V^{\otimes 2r} \otimes V^{*\otimes 2s} 
      \to \Lambda_r(\Lambda_2(V)) \otimes \Lambda_s(S_2(V^*)), \\
      e_{i_1} \otimes e_{j_1} \otimes \cdots \otimes e_{i_r} \otimes e_{j_r}
      \otimes 
      e^*_{k_1} \otimes e^*_{l_1} \otimes \cdots \otimes e^*_{k_s} \otimes e^*_{l_s}
      \mapsto
      a_{i_1j_1} \cdots a_{i_rj_r}
      b_{k_1l_1} \cdots b_{k_sl_s}.
   \end{gather*}
   Thus any $GL(V)$-invariant in $\Lambda_r(\Lambda_2(V)) \otimes \Lambda_s(S_2(V^*))$
   comes from a $GL(V)$-invariants in $V^{\otimes 2r} \otimes V^{* \otimes 2s}$,
   and we see $GL(V)$-invariants in $V^{\otimes 2r} \otimes V^{* \otimes 2s}$ 
   by the first fundamental theorem of invariant theory for vector invariants (\cite{W}, \cite{GW}).
   Indeed, when $r \ne s$, we have $(V^{\otimes 2r} \otimes V^{*\otimes 2s})^{GL(V)} = \{ 0 \}$.
   When $r = s$, any $GL(V)$-invariant in $V^{\otimes 2r} \otimes V^{*\otimes 2s}$
   can be expressed as a linear combination of elements in the form
   $$
      \sum_{1 \leq i_1,\ldots,i_{2r} \leq n}
      e_{i_1} \otimes  e_{i_2} \otimes \cdots \otimes e_{i_{2r-1}} \otimes e_{i_{2r}} \otimes 
      e^*_{i_{\sigma(1)}} \otimes e^*_{i_{\sigma(2)}} \otimes \cdots 
         \otimes e^*_{i_{\sigma(2r-1)}} \otimes e^*_{i_{\sigma(2r)}}
   $$
   with $\sigma \in S_{2r}$.
   The image of this element is equal to 
   $$
      \sum_{1 \leq i_1,\ldots,i_{2r} \leq n}
      a_{i_1i_2} \cdots a_{i_{2r-1}i_{2r}}
      b_{i_{\sigma(1)}i_{\sigma(2)}} \cdots b_{i_{\sigma(2r-1)}i_{\sigma(2r)}},
   $$
   but this is equal to $0$ as seen from Proposition~\ref{prop:tr(AB)^l}.
\end{proof}

\subsection{Cayley--Hamilton type theorem for {\boldmath$\Lambda(\Lambda_2(V) \oplus S_2(V^*))$}}
\label{subsec:CH_type_thm3}
For the matrices $A$ and $B$, we have the following relation:

\begin{theorem}\label{thm:CH_type_thm3} \sl
   We have the following relation in 
   $\operatorname{Mat}_{n,n}(\Lambda(\Lambda_2(V) \oplus S_2(V^*)))${\rm :}
   $$
      (AB)^{n-1}A = 0.
   $$
\end{theorem}

\begin{remark}
   We have $(AB)^{n-1} \ne 0$,
   because the coefficient of $a_{12} b_{22} a_{23} b_{33} \cdots a_{n-1,n} b_{nn}$
   in the $(1,n)$th entry of $(AB)^{n-1}$ is equal to $1$.
   Similarly we have $(BA)^{n-1} \ne 0$.
   Theorem~\ref{thm:CH_type_thm3} is best possible in this sense.
\end{remark}

Theorem~\ref{thm:new_AL_type_thm} follows from this Theorem~\ref{thm:CH_type_thm3}
in a way similar to the proof of Theorem~\ref{thm:AL_thm}.
It is natural to regard this Theorem~\ref{thm:CH_type_thm3} 
as a Cayley--Hamilton type theorem.

To prove Theorem~\ref{thm:CH_type_thm3}, we put 
\begin{align*}
   Q &= D([AB]^{n-1} \,|\, A \,|\, w,v) \\
   &= 
   \sum_{\sigma \in S_{n+1}} \sum_{1 \leq i_1,\ldots,i_{n+1} \leq n} 
   \operatorname{sgn}(\sigma)
   (AB)_{i_1 i_{\sigma(1)}} \cdots (AB)_{i_{n-1} i_{\sigma(n-1)}} 
   A_{i_n i_{n+1}}
   w_{i_{\sigma(n)}} v_{i_{\sigma(n+1)}}.
\end{align*}
Here $v_1,\ldots,v_n$, $w_1,\ldots,w_n$ are arbitrary complex numbers,
and we put $v = {}^t(v_1,\ldots,v_n)$ and $w = {}^t(w_1,\ldots,w_n)$.
Since we have $Q = 0$ as before,
Theorem~\ref{thm:CH_type_thm3} follows from the following relation:

\begin{proposition}\label{prop:equivalent_with_CH3}\sl
   We have
   $$
      Q = (-)^{n-1} 2 (n-1)! \,{}^tw A(BA)^{n-1} v.
   $$
\end{proposition}

Let us prove Proposition~\ref{prop:equivalent_with_CH3}.
We have the following recurrence relations (Laplace type expansions):

\begin{lemma}\sl
   We have
   \begin{align*}
      D([AB]^r,(AB)^s) 
      &= D([AB]^r) \operatorname{tr}(AB)^s 
      - r D([AB]^{r-1},(AB)^{s+1}), \\
      D([AB]^r \,|\, A(BA)^s v \,|\, w)
      &= D([AB]^r) \,{}^tw A(BA)^s v
      - r D([AB]^{r-1} \,|\, A(BA)^{s+1} v \,|\, w), \\
      D([AB]^r \,|\, A \,|\, w, (BA)^s v)
      &= 2D([AB]^r \,|\, A(BA)^s v \,|\, w)
      - r D([AB]^{r-1} \,|\, A \,|\, w, (BA)^{s+1} v).
   \end{align*}
\end{lemma}

Using this and Proposition~\ref{prop:tr(AB)^l},
we have the following relations by induction on $r$:

\begin{lemma}\sl 
   We have
   \begin{align*}
      D([AB]^r) &= \delta_{r0}, \\
      D([AB]^r \,|\, A(BA)^s v \,|\, w)
      &= (-)^r r! \,{}^tw A(BA)^{r+s} v, \\
      D([AB]^r \,|\, A \,|\, w, (BA)^s v)
      &= (-)^r 2 r! \,{}^tw A(BA)^{r+s} v.
   \end{align*}
\end{lemma}

Proposition~\ref{prop:equivalent_with_CH3} can be regarded 
as a special case of the last relation in this lemma.

%
\section*{Acknowledgements}
%
%
The author is grateful for the hospitality of 
the Max-Planck-Institut f\"ur Mathematik in Bonn, where
the essential part of this work started.
The author is also grateful to the referee for helpful comments and suggestions.

%
%
%

\end{document}